\documentclass{article}

\usepackage{amsmath,amssymb,amsbsy,amsfonts,amsthm,latexsym,
            amsopn,amstext,amsxtra,amscd,stmaryrd,fullpage}

\usepackage[mathscr]{eucal}
\usepackage{comment}
\usepackage{color}
\begin{document}

\newtheorem{theorem}{Theorem}
\newtheorem{lemma}[theorem]{Lemma}
\newtheorem{corollary}[theorem]{Corollary}
\newtheorem{proposition}[theorem]{Proposition}

\theoremstyle{definition}
\newtheorem*{definition}{Definition}
\newtheorem*{remark}{Remark}
\newtheorem*{example}{Example}


\def\cA{\mathcal A}
\def\cB{\mathcal B}
\def\cC{\mathcal C}
\def\cD{\mathcal D}
\def\cE{\mathcal E}
\def\cF{\mathcal F}
\def\cG{\mathcal G}
\def\cH{\mathcal H}
\def\cI{\mathcal I}
\def\cJ{\mathcal J}
\def\cK{\mathcal K}
\def\cL{\mathcal L}
\def\cM{\mathcal M}
\def\cN{\mathcal N}
\def\cO{\mathcal O}
\def\cP{\mathcal P}
\def\cQ{\mathcal Q}
\def\cR{\mathcal R}
\def\cS{\mathcal S}
\def\cU{\mathcal U}
\def\cT{\mathcal T}
\def\cV{\mathcal V}
\def\cW{\mathcal W}
\def\cX{\mathcal X}
\def\cY{\mathcal Y}
\def\cZ{\mathcal Z}


\def\sA{\mathscr A}
\def\sB{\mathscr B}
\def\sC{\mathscr C}
\def\sD{\mathscr D}
\def\sE{\mathscr E}
\def\sF{\mathscr F}
\def\sG{\mathscr G}
\def\sH{\mathscr H}
\def\sI{\mathscr I}
\def\sJ{\mathscr J}
\def\sK{\mathscr K}
\def\sL{\mathscr L}
\def\sM{\mathscr M}
\def\sN{\mathscr N}
\def\sO{\mathscr O}
\def\sP{\mathscr P}
\def\sQ{\mathscr Q}
\def\sR{\mathscr R}
\def\sS{\mathscr S}
\def\sU{\mathscr U}
\def\sT{\mathscr T}
\def\sV{\mathscr V}
\def\sW{\mathscr W}
\def\sX{\mathscr X}
\def\sY{\mathscr Y}
\def\sZ{\mathscr Z}


\def\fA{\mathfrak A}
\def\fB{\mathfrak B}
\def\fC{\mathfrak C}
\def\fD{\mathfrak D}
\def\fE{\mathfrak E}
\def\fF{\mathfrak F}
\def\fG{\mathfrak G}
\def\fH{\mathfrak H}
\def\fI{\mathfrak I}
\def\fJ{\mathfrak J}
\def\fK{\mathfrak K}
\def\fL{\mathfrak L}
\def\fM{\mathfrak M}
\def\fN{\mathfrak N}
\def\fO{\mathfrak O}
\def\fP{\mathfrak P}
\def\fQ{\mathfrak Q}
\def\fR{\mathfrak R}
\def\fS{\mathfrak S}
\def\fU{\mathfrak U}
\def\fT{\mathfrak T}
\def\fV{\mathfrak V}
\def\fW{\mathfrak W}
\def\fX{\mathfrak X}
\def\fY{\mathfrak Y}
\def\fZ{\mathfrak Z}


\def\C{{\mathbb C}}
\def\F{{\mathbb F}}
\def\K{{\mathbb K}}
\def\L{{\mathbb L}}
\def\N{{\mathbb N}}
\def\Q{{\mathbb Q}}
\def\R{{\mathbb R}}
\def\Z{{\mathbb Z}}
\def\E{{\mathbb E}}
\def\T{{\mathbb T}}
\def\P{{\mathbb P}}
\def\D{{\mathbb D}}


\def\eps{\varepsilon}
\def\mand{\qquad\mbox{and}\qquad}
\def\\{\cr}
\def\({\left(}
\def\){\right)}
\def\[{\left[}
\def\]{\right]}
\def\<{\langle}
\def\>{\rangle}
\def\fl#1{\left\lfloor#1\right\rfloor}
\def\rf#1{\left\lceil#1\right\rceil}
\def\le{\leqslant}
\def\ge{\geqslant}
\def\ds{\displaystyle}

\def\xxx{\vskip5pt\hrule\vskip5pt}
\def\yyy{\vskip5pt\hrule\vskip2pt\hrule\vskip5pt}
\def\imhere{ \xxx\centerline{\sc I'm here}\xxx }

\newcommand{\comm}[1]{\marginpar{
\vskip-\baselineskip \raggedright\footnotesize
\itshape\hrule\smallskip#1\par\smallskip\hrule}}


\def\e{\mathbf{e}}
\def\sPrc{{\displaystyle \sP_r^{(c)}}}

\title{\bf  The Variance of the Number of Zeros for Complex Random Polynomials Spanned by OPUC
}

\author{
{\sc Aaron M.~Yeager} }

\date{}
\newcommand{\Addresses}{{
  \bigskip
  \footnotesize

  Aaron M.~Yeager, \textsc{Department of Mathematics, College of Coastal Georgia,
    Brunswick, Georgia 31520}\par\nopagebreak
  \textit{E-mail address}: \texttt{ayeager@ccga.edu}
}}

\maketitle
\begin{abstract}
Let $\{\varphi_k\}_{k=0}^\infty $ be a sequence of orthonormal polynomials on the unit circle (OPUC) with respect to a probability measure $ \mu $. We study the variance of the number of zeros of random linear combinations of the form
$$
P_n(z)=\sum_{k=0}^{n}\eta_k\varphi_k(z),
$$
where $\{\eta_k\}_{k=0}^n $ are complex-valued random variables.  Under the assumption that the distribution for each $\eta_k$ satisfies certain uniform bounds for the fractional and logarithmic moments, for the cases when $\{\varphi_k\}$ are regular in the sense of Ullman-Stahl-Totik or are such that the measure of orthogonality $\mu$ satisfies $d\mu(\theta)=w(\theta)d\theta$ where $w(\theta)=v(\theta)\prod_{j=1}^J|\theta - \theta_j|^{\alpha_j}$, with $v(\theta)\geq c>0$, $\theta,\theta_j\in [0,2\pi)$, and $\alpha_j>0$, we give a quantitative estimate on the the variance of the number of zeros of $P_n$ in sectors that intersect the unit circle.  When $\{\varphi_k\}$ are real-valued on the real-line from the Nevai class and $\{\eta_k\}$  are i.i.d.~complex-valued standard Gaussian, we prove a formula for the limiting value of variance of the number of zeros of $P_n$ in annuli that do not contain the unit circle.
\end{abstract}



\section{Introduction}

A \emph{random polynomial} is a polynomial of the form
\begin{equation*}\label{algPn}
p_n(z)=\eta_nz^n+\eta_{n-1}z^{n-1}+\cdots +\eta_1 z + \eta_0
\end{equation*}
where $\{\eta_j\}$ are random variables.  The systematic study of zeros of random polynomials has a rich history dating back the classical works starting in the 1930's (c.f. Bloch and P\'{o}lya \cite{BP}, Littlewood and Offord (\cite{LO1}, \cite{LO2}, \cite{LO3}, \cite{LO4}, \cite{LO5})).  For a summary of the early results in the field we refer the reader to the books by Bharucha-Reid and M. Sambandham \cite{BR}, and Farahmond \cite{Fa}.

Let $N_n(\Omega)$ denote the number of zeros of $p_n$ in $\Omega\subset \C$, and $\E$ be the expected value.  When  $\{\eta_j\}$ are i.i.d.~standard  Gaussian, in 1943 Kac \cite{K1} (and independently Rice \cite{R} in 1945) produced an integral equation for  $\E[N_n(\Omega)]$, with $\Omega \subset \R$.
Kac went further to establish the asymptotic
\begin{equation*}
\E[N_n(\R)] = \frac{2+o(1)}\pi\log n \quad \text{as} \quad n\to\infty.
\end{equation*}
The error term in the above asymptotic was later sharpened by Hammersley \cite{HM}, Jamrom (\cite{JA1}, \cite{JA2}), Wang \cite{WG}, Edelman and Kostlan \cite{EK}, and finally Wilkins \cite{WL} who showed that
$$
\E[N_n(\R)] \sim \frac2\pi\log n + \sum_{k=0}^\infty A_kn^{-k},
$$
where $ \{A_k\} $ are constants.

We denote the variance of the number of zeros of a random polynomial in $\Omega \subset \C$ as
\begin{equation}\label{vardef}
\text{Var}[N_n(\Omega)]:=\E[N_n(\Omega)^2]-\E[N_n(\Omega)]^2.
\end{equation}
The first result concerning the variance of the number of real zeros of a random algebraic polynomial with i.i.d.~real-valued standard Gaussian coefficients was an upper bound provided by Stevens \cite{Stevens2} in 1965.  Specifically, in this case he gave the upper bound
$$\textup{Var}[N_n(\R)]<32\E[N_n(\R)]+2.5+(\log n)^2/\sqrt{n}, \quad \text{for} \quad n\geq 32.$$
Soon after, in 1968 Fairly \cite{FLY} computed the exact variances in this case and in the case with the coefficients of the random algebraic polynomial take the values $\pm 1$ with equal probabilities for polynomials of degree up to $11$.

In 1974, Maslova \cite{MAS1} considered the case when the random algebraic polynomial has i.i.d.~real-valued coefficients $\{\eta_k\}$ such that $\P[\eta_k=0]=0$, $\E[\eta_k]=0$, and $\E[|\eta_k|^{2+s}|<\infty$ for some $s>0$.  For this case she established that as $n\rightarrow \infty$ it follows that
\begin{equation*}
\textup{Var}[N_n(\R)]\sim \frac{4}{\pi}\left(1-\frac{2}{\pi}\right)\log n, \quad \text{and}\quad 
\frac{N_n(\R)-\E[N_n(\R)]}{\sqrt{\textup{Var}[N_n(\R)] }} \buildrel d \over \rightarrow N(0,1), \quad \text{as}\quad  n\rightarrow \infty,
\end{equation*}
where $d$ denotes convergence in distribution.
We note that Nguyen and Vu \cite{NVV} have recently generalized Maslova's results to hold under the assumption that the distribution of coefficients $\{\eta_k\}$ are independent (but not necessarily identically distributed) and have  moderate growth.

In this work we study the variance of the number of zeros in $\Omega \subset \C$ for
$$P_n(z)=\sum_{k=0}^n\eta_k\varphi_k(z),$$
where $\{\eta_k\}$ are complex-valued random variables, and $\{\varphi_k\}$ are orthogonal polynomials on the unit circle (OPUC).  The OPUC are polynomials $\{\varphi_k\}$ that are defined by a probability Borel  measure $\mu$ on $\T$ such that
\begin{equation*}
\int_{\T}\varphi_n(e^{i\theta})\overline{\varphi_m(e^{i\theta})}\ d\mu(e^{i\theta})=\delta_{nm}, \quad \text{for all $n,m\in \N \cup \{0\}$}.
\end{equation*}
We note that the measure $\mu$ above is referred to as the measure of orthogonality for $\{\varphi_k\}$.
Observe that OPUC are a direct generalization of the monomials $\{z^k\}$ which have $d\mu(\theta)=d\theta/(2\pi)$.

As the topics in this paper do not cover trigonometric random polynomials, i.e. $\sum_{k=0}^n\eta_k \cos (kx)$, we note that  asymptotics for the variance of the number of real zeros in $[0,2\pi]$ has been well studied (cf. Boomolny, Bohigas, Leboeuf \cite{BBL}, Farahmand \cite{F3}, Grandville and Wigman \cite{GW}, and Su and Shao \cite{SS}).  Similarly we  mention the works of Forrester and Honner \cite{FH}, Hannay \cite{JH}, Shiffman and Zeldtich \cite{SHZ4}, Bleher and Di \cite{BD2}, that concern asymptotics for variance of the number of zeros for weighted random polynomials, i.e. random polynomials of the form $\sum_{k=0}^n \eta_kc_k z^k$ where either $c_k={n\choose k}^{1/2}$ or $c_k=1/k!$.

\section{Main Results}\label{SVCK}

\subsection{Variance of the Number of Zeros in Sectors that intersect the Unit Circle}
 Let
$$A_r(\alpha,\beta)=\{z\in \C : \alpha \leq \arg z<\beta\leq 2\pi , \ 1/r<|z|<r,\  0<r<1 \}.$$

Consider the random orthogonal polynomial
\begin{equation}\label{PNNG}
P_n(z)=\sum_{k=0}^n\eta_{k,n}\varphi_k(z),
\end{equation}
where $\{\varphi_k\}$ are OPUC, and  $\{\eta_{k,n}\}$ are complex-valued random variables such that
\begin{equation}\label{hyp3}
\sup \{\E[|\eta_{k,n}|^t]\ | \ k=0,1,\dots,n, \ n\in \N\}<\infty,\quad  t\in(0,1],
\end{equation}
and
\begin{equation}\label{hyp4}
\min\left( \inf_{n\in \N}\E[\log |\eta_{n,n}|], \ \inf_{n\in \N, z\in \C}\E[\log |\eta_{0,n}+z|] \right)>-\infty.
\end{equation}
We do not assume that random variables $\{\eta_{k,n}\}$ are independent nor identically distributed.  In the case when the random variables are independent, then all uniform bounds in \eqref{hyp3} and \eqref{hyp4} reduce to those of a single random variable $\eta_{0,0}$.  We note that the assumption \eqref{hyp3} is to ensure that the tails of the distribution have sufficient decay, and the assumption \eqref{hyp4} is present so that that random coefficients $\eta_{n,n}$ and $\eta_{0,n}$ do not vanish so often that random orthogonal polynomial has degree less than $n$ or factors as $z$ times a degree $n-1$ polynomial.  Such assumptions hold for a wide class of random variables, which certainly includes Gaussian random variables.


Applying Theorem 3.1 of \cite{PritYgr15} and further estimating we are able to obtain the following:

\begin{lemma} \label{cor3.2}
For the random orthogonal polynomial \eqref{PNNG} with the
uniform estimates \eqref{hyp3} and \eqref{hyp4} on the coefficients $\{\eta_{k,n}\}$, suppose that the measure of orthogonality $\mu$ associated to $\{\varphi_k\}$ is regular in the sense of Ullman-Stahl-Totik, that is,
\begin{equation}\label{USThyp}
\varepsilon_n:=\frac{1}{n}\log |\kappa_n|\rightarrow 0, \quad \text{as} \quad n\rightarrow \infty,
\end{equation}
where $\kappa_n$ is the leading coefficient of $\varphi_n$.  Then
\begin{align} \label{3.8}
\E\left[\left|\frac{N_n(A_r(\alpha,\beta))}{n}-\frac{\beta-\alpha}{2\pi}\right|\right] = \mathcal O \left( \max \left\{ \sqrt{\frac{\log n}{n} },\varepsilon_n^{1/4} \right\} \right), \quad \text{as}  \quad n\to\infty.
\end{align}
\end{lemma}

Due to Corollary 3.2 in \cite{PritYgr15} (and the comment preceding this Corollary), it was shown that if the measure of orthogonality $\mu$ associated to the OPUC satisfies
$d\mu(\theta)=w(\theta)\,d\theta$, where $w(\theta)$ is a generalized Jacobi weight of the form
\begin{equation}\label{gjw}
w(\theta)=v(\theta)\prod_{j=1}^J|\theta-\theta_j|^{\alpha_j},
\end{equation}
with $v(\theta)\ge c > 0$, $ \theta, \theta_j\in[0,2\pi)$, and $\alpha_j>0$, then
\begin{align} \label{3.7}
\E\left[\left|\frac{N_n(A_r(\alpha,\beta))}{n}-\frac{\beta-\alpha}{2\pi}\right|\right] = \mathcal O\left(\sqrt{\frac{\log{n}}{n}}\right), \quad \text{as} \quad n\to\infty.
\end{align}

We apply \eqref{3.8} and \eqref{3.7} for an estimate of the variance of the number of zeros.
\begin{theorem}\label{VarAnnOPUC}
The random orthogonal polynomial \eqref{PNNG} with $\{\eta_{k,n}\}$  complex-valued random variables satisfying conditions \eqref{hyp3} and \eqref{hyp4} possesses the following properties:
\begin{enumerate}
\item If the measure of orthogonality $\mu$ associated to $\{\varphi_k\}$ satisfies $d\mu(\theta)=w(\theta)d\theta$, where $w(\theta)$ is given by \eqref{gjw}, we have
\begin{equation}\label{asvar}
\frac{\textup{Var}[N_n(A_r(\alpha,\beta))]}{n^2}=\mathcal O \left( \sqrt{\frac{\log n}{n}} \right), \quad \text{as} \quad n\rightarrow \infty.
\end{equation}
\item Under the assumption that $\{\varphi_k\}$ are regular in the sense of Ullman-Stahl-Totik, it follows that
\begin{equation}\label{ustvar}
\frac{\textup{Var}[N_n(A_r(\alpha,\beta))]}{n^2}=\mathcal O \left( \max \left\{ \sqrt{\frac{\log n}{n} },\varepsilon_n^{1/4} \right\} \right), \quad \text{as} \quad n\rightarrow \infty,
\end{equation}
where $\varepsilon_n$ is given by \eqref{USThyp}.
\end{enumerate}
\end{theorem}

\subsection{Variance of the Number of Zeros in Annuli not intersecting the Unit Circle}

We now consider random orthogonal polynomials
\begin{equation}\label{ropucn}
P_n(z)=\sum_{k=0}^n\eta_k\varphi_k(z),
 \end{equation}
 where $\{\eta_k\}$ are i.i.d.~complex-valued standard Gaussian, and $\{\varphi_k\}$ are OPUC that are real-valued on the real-line and from the Nevai Class.  We remind the reader that the Nevai Class of OPUC are the collection of OPUC $\{\varphi_k\}$ with property that locally uniformly for $z\in \D$ we have
\begin{equation*}
\lim_{n\rightarrow \infty}\frac{\varphi_n(z)}{\varphi_n^*(z)}=0,\quad \text{with} \quad \varphi_n^*(z)=z^n\overline{\varphi_n(1/\overline{z})}.
\end{equation*}
We remark that our  assumption that $\{\varphi_k\}$ are real-valued on the real-line is to ensure that
locally uniformly for $z\in \C\setminus \overline{\D}$ it follows that
\begin{equation*}
\lim_{n\rightarrow \infty}\frac{\varphi_n^*(z)}{\varphi_n(z)}=0.
\end{equation*}

In this setting we will study the variance of the number of zeros via examining the first and second correlation functions.  We denote the first and second correlation functions for $P_n(z)$ as $\rho_n^{(1)}(z)$ and $\rho_n^{(2)}(z,w)$, respectively. To see the connection between the variance of the number of zeros and these correlation functions, observe that for a measurable set $\Omega\subset \C$ it follows that
\begin{align}\nonumber
\text{Var}[N_n(\Omega)]&=\E[(N_n(\Omega))^2]-\left(\E[N_n(\Omega)]\right)^2\\
\nonumber
&=\E[N_n(\Omega)]+\E[N_n(\Omega)(N_n(\Omega)-1)]-\left(\E[N_n(\Omega)]\right)^2\\
\label{varform}
&=\int_{\Omega}\rho^{(1)}_n(z)\ dA(z)+\int_{\Omega}\int_{\Omega}\rho^{(2)}_n(z,w)\ dA(z) \ dA(w)
-  \int_{\Omega}\int_{\Omega}\rho^{(1)}_n(z)\rho^{(1)}_n(w)\ dA(z) \ dA(w),
\end{align}
where the equalities
$$\E[N_n(\Omega)]=\int_{\Omega} \rho_n^{(1)}(z)\ dA(z) \quad \text{and} \quad \E[N_n(\Omega)(N_n(\Omega)-1)]=\int_{\Omega}\int_{\Omega}\rho^{(2)}_n(z,w)\ dA(z) \ dA(w)$$
are known results.

For the random orthogonal polynomial \eqref{ropucn}, due to Corollary 2.2 from \cite{AY2}  we have
\begin{equation}\label{fstinop}
\lim_{n\rightarrow \infty}\rho_n^{(1)}(z)=\frac{1}{\pi(1-|z|^2)^2}
\end{equation}
locally uniformly for all $z\in \C \setminus \T$.  While the formula for the second correlation function associated to the random orthogonal polynomial \eqref{ropucn} is rather complicated, its limit as $n$ tends to infinity has a striking simplicity.
\begin{theorem}\label{SECINTOP}
When $z$ and $w$ are both in the unit disk or both in the the exterior of the unit disk, the second correlation function for the random orthogonal polynomial \eqref{ropucn} satisfies
\begin{equation}\label{secintop}
\lim_{n\rightarrow \infty}\rho^{(2)}_n(z,w)=\frac{1}{\pi^2}\left(\frac{1}{(1-|z|^2)^2(1-|w|^2)^2}-\frac{1}{|1-z\overline{w}|^4}\right),
\end{equation}
where the above convergence takes place locally uniformly.
\end{theorem}

Our next theorem gives the limiting value of the variance of the number of zeros of the random orthogonal polynomial \eqref{ropucn} in an annulus
$$A(s,t)=\{z\in \C : 0\leq s < |z|< t\},$$
that does not contain the unit disk.  We remark that our approach allows one to similarly consider sectors that do not contain the unit circle at the expense the below result taking a complicated shape.

\begin{theorem}\label{VAROPS}
The random orthogonal polynomial \eqref{ropucn}  possesses the property that
\begin{align*}
\lim_{n\rightarrow \infty}\textup{Var}[N_n(A(s,t))]= \begin{cases}
\displaystyle\frac{(t^2-s^2)[1-s^2(t^4(2+s^2)-2)]}{(1-t^4)(1-s^4)(1-(st)^2)}, & A(s,t)\subsetneq \D, \\[3ex]
\displaystyle\frac{(t^2-s^2)[1-t^2(s^4(2+t^2)-2)]}{(1-t^4)(1-s^4)(1-(st)^2)}, & A(s,t) \subsetneq \C\setminus \overline{\D}.
\end{cases}
\end{align*}
\end{theorem}
We note that taking $s=0$ and $t<1$ in the above theorem, we achieve that the random orthogonal polynomial \eqref{ropucn} gives the simple formula:
\begin{equation*}
\lim_{n\rightarrow \infty}\textup{Var}[N_n(D(0,t))]=\frac{t^2}{1-t^4},
\end{equation*}
where $D(0,t)=\{z\in \C : |z|<t<1\}$.

\section{The Proofs}

\begin{proof}[Proof of Lemma \ref{cor3.2}]
Setting
\begin{equation*}
\varphi_k(z)=\kappa_{k,k}z^k+a_{k-1,k}z^{k-1}+a_{k-2,k}z^{k-2}+\cdots +a_{0,k},\quad k\in \{0,1,\dots,n\},
\end{equation*}
 by Theorem 3.1 of \cite{PritYgr15} for all large $n\in \N$ we have
\begin{align} \label{3.2}
\E\left[\left|\frac{N_n(A_r(\alpha,\beta))}{n}-\frac{\beta-\alpha}{2\pi}\right|\right]  \nonumber &\le C_r \left[\frac{1}{n}\left(\frac{1}{t}\log \left(\sum_{k=0}^n \E[|\eta_k|^t]\right) + \log \max_{0 \le k \le n} \|\varphi_k\|_{\infty} - \frac{1}{2} \E[\log|D_n|]\right)\right]^{1/2},
\end{align}
where
$$D_n := \eta_n \kappa_{n,n} \sum_{k=0}^n \eta_k a_{0,k}\quad \text{and} \quad C_r = \sqrt{\frac{2\pi}{\mathbf{k}}}+\frac{2}{1-r}. $$

Observe that the uniform bounds \eqref{hyp3} and \eqref{hyp4} on the expectations for the coefficients  immediately give that
\begin{align*}
\frac{1}{tn}\log \left(\sum_{k=0}^n \E[|\eta_{k,n}|^t]\right) = O\left(\frac{\log n}{n}\right) \quad\mbox{and}\quad \frac{1}{2n} \E[\log|D_n|] \ge \frac{1}{n} \log|\kappa_{n,n}| + O\left(\frac{1}{n}\right).
\end{align*}
Thus to establish \eqref{3.7},  it suffices to show
\begin{equation}
\frac{1}{n}\log \max_{0\leq k \leq n}\|\varphi_k\|_{\infty}=\mathcal O(\sqrt{\varepsilon_n}),
\end{equation}
where $\varepsilon_n=\log|\kappa_{n,n}|/n$.

Writing $\kappa_{k,k}=\kappa_k$, equation 1.5.22 of \cite{BS} gives
\begin{equation}\label{kapparep}
\kappa_{k}=\prod_{j=0}^{k-1}(1-|\alpha_{j}|^2)^{-1/2},
\end{equation}
where $\{\alpha_j\}\subset \D$ are recurrence coefficients coming from the three term recurrence relation (c.f. Theorem 1.5.4 \cite{BS}):
\begin{equation*}
\varphi_{j+1}(z)=\frac{ z \varphi_j(z)-\bar{\alpha}_j\varphi_{j}^*(z)}{\sqrt{1-|\alpha_j|^2}},\quad j=0,1,\dots,
\end{equation*}
with $\varphi_j^*(z)=z^j\overline{\varphi_j(1/\bar{z})}$. For the normalized OPUC, denoted as $\Phi_k(z)$, we have $\varphi_k(z)=\kappa_{k}\Phi_k(z)$, so that appealing to \eqref{kapparep} and (1.5.17) of Theorem 1.5.3 in \cite{BS} yields
\begin{align*}
\log \max_{0\leq k \leq n}\|\varphi_k\|_{\infty}&=\log \max_{0\leq k \leq n}\|\kappa_{k}\Phi_k(z)\|_{\infty}\\
&\leq \log \left( |\kappa_n| \max_{0\leq k \leq n}\|\Phi_k(z)\|_{\infty} \right)\\
&\leq \log \left(|\kappa_n|\max_{0\leq k \leq n} \exp\left( \sum_{j=0}^{k-1}|\alpha_j| \right)  \right)\\
&\leq \log \left(|\kappa_n| \exp\left( \sum_{j=0}^{n-1}|\alpha_j| \right)  \right)\\
&= \log |\kappa_n| + \sum_{j=0}^{n-1}|\alpha_j|\\
&\leq \log |\kappa_n| + \left(n\sum_{j=0}^{n-1}|\alpha_j|^2\right)^{1/2},
\end{align*}
where we have relied on the Cauchy-Schwarz inequality in the last inequality.  To estimate the second term above, notice that since each $\alpha_j \in \D$, we have
$$\log \frac{1}{1-|\alpha_j|^2}=\sum_{k=1}^{\infty}\frac{|\alpha_j|^{2k}}{k}>|\alpha_j|^2.$$
Therefore
\begin{align}\nonumber
\frac{1}{n}\log \max_{0\leq k \leq n}\|\varphi_k\|_{\infty}\leq \frac{1}{n}\left(\log |\kappa_n| +\left(n\sum_{j=0}^{n-1}\log \frac{1}{1-|\alpha_j|^2}\right)^{1/2}\right)
=\frac{1}{n}\left(\log |\kappa_n|+\left(2n\log |\kappa_n|\right)^{1/2}\right)
\leq \mathcal O \left(\sqrt{\varepsilon_n}\right),
\end{align}
where we have relied on \eqref{kapparep} for the second term in the equality,  which completes the desired estimate to give \eqref{3.8}.
\end{proof}

\begin{proof}[Proof of Theorem \ref{VarAnnOPUC}]
We first prove \eqref{asvar}.  
Observe that
\begin{align}\nonumber
\E\left[ \left| \left(\frac{N_n(A_r(\alpha,\beta))}{n}\right)^2- \left(\frac{\beta-\alpha}{2\pi}\right)^2 \right| \right]&=\E\left[ \left| \left(\frac{N_n(A_r(\alpha,\beta))}{n}-\frac{\beta-\alpha}{2\pi}\right)\left(\frac{N_n(A_r(\alpha,\beta))}{n}+\frac{\beta-\alpha}{2\pi} \right) \right| \right]\\[1ex]
\nonumber
&\leq 2\cdot \E\left[ \left| \frac{N_n(A_r(\alpha,\beta))}{n}-\frac{\beta-\alpha}{2\pi} \right| \right]\\[1ex]
\label{secmomcal}
&=\mathcal O \left( \sqrt{ \frac{\log n}{n} } \right),
\end{align}
where we have appealed to \eqref{3.7} in the last equality.
Thus
\begin{equation*}
\E\left[ \left(\frac{N_n(A_r(\alpha,\beta))}{n}\right)^2 \right]\leq  \mathcal O \left( \sqrt{ \frac{\log n}{n} } \right)+ \left(\frac{\beta-\alpha}{2\pi}\right)^2.
\end{equation*}
Therefore
\begin{align}
\nonumber
\frac{\text{Var}[N_n(A_r(\alpha,\beta))]}{n^2}&=\frac{1}{n^2}\E[N_n(A_r(\alpha,\beta))^2]-\frac{1}{n^2}\left(\E[N_n(A_r(\alpha,\beta))]  \right)^2\\[5pt]
\nonumber
&=\E\left[\left(\frac{N_n(A_r(\alpha,\beta))}{n}\right)^2\right]-\E\left[\frac{N_n(A_r(\alpha,\beta))}{n}\right]^2\\[5pt]
\nonumber
&\leq \mathcal O \left(\sqrt{\frac{\log n}{n}}\right)+\left( \frac{\beta-\alpha}{2\pi} \right)^2-\E\left[\frac{N_n(A_r(\alpha,\beta))}{n}\right]^2\\[5pt]
\nonumber
&=\mathcal O \left(\sqrt{\frac{\log n}{n}}\right)+ \left( \frac{\beta-\alpha}{2\pi}- \E\left[\frac{N_n(A_r(\alpha,\beta))}{n}\right]\right)\left( \frac{\beta-\alpha}{2\pi}+ \E\left[\frac{N_n(A_r(\alpha,\beta))}{n}\right] \right)\\[5pt]
\nonumber
&\leq \mathcal O \left(\sqrt{\frac{\log n}{n}}\right)+\E\left[ \left| \frac{N_n(A_r(\alpha,\beta))}{n}-\frac{\beta-\alpha}{2\pi} \right| \right]\left( \frac{\beta-\alpha}{2\pi}+ 1 \right)\\[5pt]
\nonumber
&=\mathcal O \left(\sqrt{\frac{\log n}{n}}\right)+\mathcal O \left(\sqrt{\frac{\log n}{n}}\right)\\[5pt]
\label{varcal}
&= \mathcal O \left(\sqrt{\frac{\log n}{n}}\right),
\end{align}
which gives the result of \eqref{asvar}.

For the case $\{\varphi_k\}$ are regular in the sense of Ullman-Stahl-Totik, as Lemma \ref{cor3.2} gave
\begin{align*}
\E\left[\left|\frac{N_n(A_r(\alpha,\beta))}{n}-\frac{\beta-\alpha}{2\pi}\right|\right] = \mathcal O \left( \max \left\{ \sqrt{\frac{\log n}{n} },\varepsilon_n^{1/4} \right\} \right) ,
\end{align*}
repeating the computation \eqref{secmomcal} we see that
\begin{equation*}
\E\left[ \left(\frac{N_n(A_r(\alpha,\beta))}{n}\right)^2 \right]\leq  \mathcal O \left( \max \left\{ \sqrt{\frac{\log n}{n} },\varepsilon_n^{1/4} \right\} \right)+ \left(\frac{\beta-\alpha}{2\pi}\right)^2.
\end{equation*}
Hence repeating the computation \eqref{varcal} we achieve
\begin{equation*}
\frac{\text{Var}[N_n(A_r(\alpha,\beta))]}{n^2}=\mathcal O \left( \max \left\{ \sqrt{\frac{\log n}{n} },\varepsilon_n^{1/4} \right\} \right),
\end{equation*}
and thus completing the desired result \eqref{ustvar}.
\end{proof}

\begin{proof}[Proof of Theorem \ref{SECINTOP}]

For the random sum $f_n(z)=\sum_{k=0}^n\eta_kp_j(z)$, where $\{\eta_j\}$ are complex valued i.i.d.~Gaussian random variables and $\{p_j(z)\}$ are a polynomial basis with the degree of $p_j(z)$ equal to $j$, Corollary 3.4.2 of \cite{ZGAF} gives the following formulas for the correlation functions:
\begin{equation}\label{fsecint1}
\rho^{(m)}_n(z_1,\dots,z_m)=\frac{\textup{Perm}(C-B^*A^{-1}B)}{\pi^m\textup{Det(A)}},
\end{equation}
where $\textup{Perm}(\cdot)$ denotes the permanent of a matrix, $B^*$ is the conjugate transpose of the matrix $B$, and
\begin{align*}
A&=\left[\E\left[p_n(z_i)\overline{p_n(z_j)}\right]\right]_{0\leq i, j\leq m}=\left[ \sum_{k=0}^np_k(z_i)\overline{p_k(z_j)} \right]_{0\leq i,j\leq m}:=[K_n(z_i,z_j)]_{0\leq i, j\leq m},\\
B&=\left[\E\left[p_n(z_i)\overline{p_n^{\prime}(z_j)}\right]\right]_{0\leq i, j\leq m}=\left[ \sum_{k=0}^np_k(z_i)\overline{p_k^{\prime}(z_j)} \right]_{0\leq i,j\leq m}:=[K^{(0,1)}_n(z_i,z_j)]_{0\leq i, j\leq m},\\
C&=\left[\E\left[p_n^{\prime}(z_i)\overline{p_n^{\prime}(z_j)}\right]\right]_{0\leq i, j\leq m}=\left[ \sum_{k=0}^np_k^{\prime}(z_i)\overline{p_k^{\prime}(z_j)} \right]_{0\leq i,j\leq m}:=[K^{(1,1)}_n(z_i,z_j)]_{0\leq i, j\leq m},
\end{align*}
with the second equality in each row above following from the property that the random variables $\{\eta_j\}$ have mean zero and variance one.  We remark that for the second correlation function with $m=2$ in the above and $z_1=z$ and $z_2=w$, in this case we have
\begin{align*}
\det A &= K_n(z,z)K_n(w,w)-K_n(z,w)K_n(w,z)\\
&=\sum_{j=0}^n|p_j(z)|^2\sum_{j=0}^n|p_j(w)|^2-\left| \sum_{j=0}^np_j(z)\overline{p_j(w)} \right|^2\\
&=\sum_{k=1}^n|p_0(z)p_k(w)-p_0(w)p_k(z)|^2+\sum_{k=2}^n|p_1(z)p_k(w)-p_1(w)p_k(z)|^2\\
&\ \ +\cdots + \sum_{k=n-2}^n|p_{n-3}(z)p_k(w)-p_{n-3}(w)p_k(z)|^2+|p_n(z)p_{n-1}(w)-p_n(w)p_{n-1}(z)|^2\\
&\geq |p_0(z)p_1(w)-p_0(w)p_1(z)|^2.
\end{align*}
As $\{p_j\}$ is a polynomial basis with $\deg p_j=j$ for all $j\in \N\cup \{0\}$, we have $p_0(z)=c$ and $p_1(z)=az +b$, for some constants $a,b,c,$ with $a,c\neq 0$.  Thus
\begin{align*}
|p_0(z)p_1(w)-p_0(w)p_1(z)|^2=|c(aw+b)-c(az+b)|^2=|ca(w-z)|^2.
\end{align*}
Hence we see that $\det A > 0$ for all $z\neq w$. As $\rho_n^{(2)}(z,z)=0$, we see that the representation
\begin{equation*}
\rho_n^{(2)}(z,w)=\frac{\textup{Perm}(C-B^*A^{-1}B)}{\pi^m\textup{Det(A)}}
\end{equation*}
is valid everywhere for all random polynomials spanned by a polynomial basis.

Expanding the permanent and the determinant in the definition of second correlation function, one sees that $\rho_n^{(2)}(z,w)$ can be written as
\begin{equation}\label{fsecint2}
\pi^2\rho_n^{(2)}(z,w)=f_n(z,w)f_n(w,z)+g_n(z,w)g_n(w,z),
\end{equation}
where
\begin{align*}
f_n(z,w)&=\frac{K_n^{(1,1)}(z,z)}{(K_n(z,z)K_n(w,w)-|K_n(z,w)|^2)^{1/2}}\\
&\quad +\frac{2\text{Re}\left( K_n(z,w)\overline{ K_n^{(0,1)}(z,z) }K_n^{(0,1)}(w,z) \right)}{(K_n(z,z)K_n(w,w)-|K_n(z,w)|^2)^{3/2}}\\
&\quad -\frac{K_n(w,w)|K_n^{(0,1)}(z,z)|^2+K_n(z,z)|K_n^{(0,1)}(w,z)|^2}{(K_n(z,z)K_n(w,w)-|K_n(z,w)|^2)^{3/2}},
\end{align*}
and
\begin{align*}
g_n(z,w)&=\frac{K_n^{(1,1)}(z,w)}{(K_n(z,z)K_n(w,w)-|K_n(z,w)|^2)^{1/2}}\\
&\quad +\frac{ K_n(z,w)\overline{ K_n^{(0,1)}(z,z) }K_n^{(0,1)}(w,w)+ \overline{K_n(z,w) K_n^{(0,1)}(w,z) }K_n^{(0,1)}(z,w)}{(K_n(z,z)K_n(w,w)-|K_n(z,w)|^2)^{3/2}}\\
&\quad -\frac{K_n(w,w)\overline{K_n^{(0,1)}(z,z)}K_n^{(0,1)}(z,w)+K_n(z,z)\overline{K_n^{(0,1)}(w,z)}K_n^{(0,1)}(w,w)}{(K_n(z,z)K_n(w,w)-|K_n(z,w)|^2)^{3/2}}.
\end{align*}

We will simplify the above formulas via the Christoffel-Darboux formula.  For a collection of OPUC $\{\varphi_j\}_{j\geq 0}$, the Christoffel-Darboux formula  (c.f.~Theorem 2.2.7, p. 124 of \cite{BS}) states that for $z,w\in \C$ with $\bar{w}z\neq 1$, we have
\begin{equation}\label{CDUP}
K_n(z,w)=\sum_{j=0}^{n}\varphi_j(z)\overline{\varphi_j(w)}= \frac{\overline{\varphi_{n+1}^{*}(w)}\varphi_{n+1}^{*}(z)-\overline{\varphi_{n+1}(w)}\varphi_{n+1}(z)}{1-\bar{w}z},
\end{equation}
where $\varphi_n^{*}(z)=z^n \overline{\varphi_n\left(\frac{1}{\bar{z}}\right)}$.


Taking the derivative of \eqref{CDUP} with respect to $\overline{w}$ yields
\begin{align}\label{k01}
K_n^{(0,1)}(z,w)&=\sum_{j=0}^{n}\varphi_j(z)\overline{\varphi_j^{\prime}(w)}=\frac{S_n(z,w)}{1-z\overline{w}}+\frac{zK_n(z,w)}{1-z\overline{w}},
\end{align}
with
\begin{equation}
S_n(z,w)=\overline{(\varphi_{n+1}^*)^{\prime}(w)}\varphi_{n+1}^*(z)-\overline{\varphi^{\prime}_{n+1}(w)}\varphi_{n+1}(z).
\end{equation}
Differentiating \eqref{k01} with respect to $z$ gives
\begin{align}\label{k11}
K_n^{(1,1)}(z,w)&=\sum_{j=0}^{n}\varphi_j^{\prime}(z)\overline{\varphi_j^{\prime}(w)}=\frac{R_n(z,w)(1-z\overline{w})+z\overline{S_n(w,z)}+\overline{w}S_n(z,w)+(1+z\overline{w})K_n(z,w)}{(1-z\overline{w})^2},
\end{align}
with
\begin{equation}
R_n(z,w)=\overline{(\varphi_{n+1}^*)^{\prime}(w)}(\varphi_{n+1}^*)^{\prime}(z)-\overline{\varphi^{\prime}_{n+1}(w)}\varphi_{n+1}^{\prime}(z).
\end{equation}

Let us rewrite \eqref{fsecint2} as
\begin{equation}
\label{fsecint3}
\pi^2 \rho_n^{(2)}(z,w)= \frac{\tilde{f}_n(z,w)\tilde{f}_n(w,z)+\tilde{g}_n(z,w)\tilde{g}_n(w,z)}{ (K_n(z,z)K_n(w,w)-|K_n(z,w)|^2)^{3}}
\end{equation}
where
\begin{align}
\label{s1}
\tilde{f}_n(z,w)&=K_n^{(1,1)}(z,z)(K_n(z,z)K_n(w,w)-|K_n(z,w)|^2)\\
\label{s2}
&\quad +2\text{Re}\left( K_n(z,w)\overline{ K_n^{(0,1)}(z,z) }K_n^{(0,1)}(w,z) \right)\\
\label{s3}
&\quad -K_n(w,w)|K_n^{(0,1)}(z,z)|^2+K_n(z,z)|K_n^{(0,1)}(w,z)|^2,
\end{align}
and
\begin{align}
\label{s4}
\tilde{g}_n(z,w)&=K_n^{(1,1)}(z,w)(K_n(z,z)K_n(w,w)-|K_n(z,w)|^2)\\
\label{s5}
&\quad + K_n(z,w)\overline{ K_n^{(0,1)}(z,z) }K_n^{(0,1)}(w,w)+ \overline{K_n(z,w) K_n^{(0,1)}(w,z) }K_n^{(0,1)}(z,w)\\
\label{s6}
&\quad -K_n(w,w)\overline{K_n^{(0,1)}(z,z)}K_n^{(0,1)}(z,w)+K_n(z,z)\overline{K_n^{(0,1)}(w,z)}K_n^{(0,1)}(w,w).
\end{align}

We now introduce the notation that $b_n(z):=\varphi_n(z)/\varphi_n^*(z)$. Observe that since the OPUC $\{\varphi_n\}$ have real coefficients, we have $b_n^{-1}(z)=b_n(1/z)$.  Thus the condition of the OPUC being from the Nevai class can be restated as
\begin{align}\label{phiasy}
\begin{cases}
      b_n(z)\rightarrow 0, &\text{locally uniformly in $|z|<1$,} \\
      b_n^{-1}(z)\rightarrow 0, &\text{locally uniformly in $|z|>1$.}
   \end{cases}
\end{align}
Consequently
\begin{align}\label{dphiasy}
\frac{\varphi_n^{\prime}(z)\varphi_n^*(z)-\varphi_n(z)(\varphi_n^*)^{\prime}(z)}{\phi_n^2(z)}=\begin{cases}
      b_n^{\prime}(z)\rightarrow 0, &\text{locally uniformly in $|z|<1$,} \\
      -(b_n^{-1})^{\prime}(z)\rightarrow 0, &\text{locally uniformly in $|z|>1$,}
   \end{cases}
\end{align}
where
\begin{align*}
\phi_n(z):=\begin{cases}
      \varphi_n^*(z), & |z|<1, \\
      \varphi_n(z), & |z|>1.
   \end{cases}
\end{align*}
Hence appealing to the above and \eqref{CDUP}, we see that as $n\rightarrow \infty$ the denominator of $\pi\rho_n^{(2)}(z,w)$ is
\begin{align}
\nonumber
(K_n(z,z)K_n(w,w)-&|K_n(z,w)|^2)^3\\
\label{rhodeom}
&=|\phi_{n+1}(z)\phi_{n+1}(w)|^6\left[\left(\frac{1}{(1-|z|^2)(1-|w|^2)}-\frac{1}{|1-z\overline{w}|^2} \right)^3+o(1)\right].
\end{align}

We now find the asymptotic for $\tilde{f}_n(z,w)$.  Set
$$S_1(z,w)= \eqref{s1}, \quad S_2(z,w) = \eqref{s2}, \quad  S_3(z,w) = \eqref{s3}.$$
 Using \eqref{k01} and \eqref{k11} we see that
\begin{align}
\label{S1}
S_1(z,w)&=\frac{(1+|z|^2)K_n(z,z)}{(1-|z|^2)^2}(K_n(z,z)K_n(w,w)-|K_n(z,w)|^2)\\
\label{S2}
&\quad +\frac{2\textup{Re}\left(\overline{z}S_n(z,z)\right)}{(1-|z|^2)^2}(K_n(z,z)K_n(w,w)-|K_n(z,w)|^2)\\
\label{S3}
&\quad + \frac{R_n(z,z)}{1-|z|^2}(K_n(z,z)K_n(w,w)-|K_n(z,w)|^2),\\
\label{S4}
S_2(z,w)&=-\frac{|w|^2K_n(z,z)|K_n(z,w)|^2}{|1-z\overline{w}|^2}-\frac{|z|^2K_n(z,z)^2K_n(w,w)}{(1-|z|^2)^2}\\
\label{S5}
&\quad -\frac{2K_n(z,z)\textup{Re}\left( \overline{w}S_n(w,z)K_n(z,w) \right)}{|1-z\overline{w}|^2}-\frac{2K_n(z,z)K_n(w,w)\textup{Re}\left(\overline{z}S_n(z,z)  \right)}{(1-|z|^2)^2}\\
\label{S6}
&\quad -\frac{K_n(z,z)|S_n(w,z)|^2}{|1-z\overline{w}|^2}-\frac{K_n(w,w)|S_n(z,z)|^2}{(1-|z|^2)^2},\\
\label{S7}
S_3(z,w)&=\frac{K_n(z,z)|K_n(z,w)|^2}{1-|z|^2}\left( \frac{1-|zw|^2}{|1-z\overline{w}|^2}-1 \right)\\
\label{S81}
&\quad +\frac{2|K_n(z,w)|^2}{1-|z|^2}\textup{Re}\left(\frac{\overline{w}S_n(z,z)}{1-z\overline{w}}  \right)\\
\label{S82}
&\quad +\frac{2K_n(z,z)}{1-|z|^2}\textup{Re}\left(\frac{\overline{z}K_n(z,w)S_n(w,z)}{1-\overline{z}w}  \right)\\
\label{S9}
&\quad + \frac{2}{1-|z|^2}\textup{Re}\left( \frac{K_n(z,w)\overline{S_n(z,z)}S_n(w,z)}{1-\overline{z}w} \right),
\end{align}
where we have made use of the identity $2\textup{Re}(z\overline{w})=1+|zw|^2-|1-z\overline{w}|^2$ to get the expression in parentheses of \eqref{S7} in the shape it is in.

We now define
\begin{align*}
\Sigma_{n,1}(z,w)&:=\eqref{S1}+\eqref{S4}+\eqref{S7}\\
\Sigma_{n,2}(z,w)&:= \eqref{S2}+\eqref{S5}+\eqref{S81}+\eqref{S82}\\
\Sigma_{n,3}(z,w)&:= \eqref{S3}+\eqref{S6}+\eqref{S9}.
\end{align*}

Simplifying and then appealing to \eqref{CDUP}, \eqref{phiasy}, and \eqref{dphiasy}, we see that
\begin{align}
\nonumber
&\Sigma_{n,1}(z,w)=\frac{K_n(z,z)^2K_n(w,w)}{(1-|z|^2)^2}-\frac{K_n(z,z)|K_n(z,w)|^2}{1-|z|^2}\left( \frac{2}{1-|z|^2}+\frac{|w|^2-1}{|1-z\overline{w}|^2}  \right)\\
\nonumber
&\quad =|\phi_{n+1}(z)|^4|\phi_{n+1}(w)|^2\Bigg[ \frac{1}{(1-|z|^2)^4(1-|w|^2)}
-\frac{1}{(1-|z|^2)|1-z\overline{w}|^2}\left( \frac{2}{1-|z|^2}+\frac{|w|^2-1}{|1-z\overline{w}|^2}  \right)+o(1) \Bigg]\\
\label{Sig1}
&\quad =|\phi_{n+1}(z)|^4|\phi_{n+1}(w)|^2\left( \frac{|z-w|^4}{(1-|z|^2)^4(1-|w|^2)|1-z\overline{w}|^4}+o(1) \right).
\end{align}

Turning now  to the asymptotic for $\Sigma_{n,2}(z,w)$, observe the first summand of \eqref{S2} and second summand of \eqref{S5} cancel algebraically, and that the sum of the second summand in \eqref{S2} and the first summand of \eqref{S81} simplify to
\begin{align}
\label{Sig21}
2\frac{|K_n(z,w)|^2}{(1-|z|^2)^2}\textup{Re}\left( \frac{(\overline{w}-\overline{z})S_n(z,z)}{1-z\overline{w}} \right).
\end{align}
The sum of the first summand of \eqref{S5} and  \eqref{S82} collect to give
\begin{align}\label{Sig22}
\frac{2K_n(z,z)}{(1-|z|^2)|1-z\overline{w}|^2}\textup{Re}\left( (\overline{z}-\overline{w})K_n(z,w)S_n(w,z) \right).
\end{align}
Combining \eqref{Sig21} with \eqref{Sig22} then appealing to \eqref{CDUP}, \eqref{k01}, \eqref{k11} to simplify the expressions, and finally using the asymptotics of \eqref{phiasy} and \eqref{dphiasy} on sees
\begin{align}
\nonumber
\Sigma_{n,2}(z,w)&=\frac{1}{(1-|z|^2)^2|1-z\overline{w}|^2}\\
\nonumber
&\quad \cdot 2\textup{Re}\Big[ (\overline{w}-\overline{z})K_n(z,w)\left(\overline{\varphi_{n+1}(z)(\varphi_{n+1}^*)^{\prime}(z)}-\overline{\varphi_{n+1}^{\prime}(z)\varphi_{n+1}^*(z)}\right)\\
\nonumber
&\qquad \qquad \cdot (\varphi_{n+1}(z)\varphi_{n+1}^*(w)-\varphi_{n+1}^*(z)\varphi_{n+1}(w)) \Big]\\
\label{Sig2}
&=o(|\phi_{n+1}(z)|^4|\phi_{n+1}(w)|^2).
\end{align}

For the sum $\Sigma_{n,3}(z,w)$, first notice using \eqref{CDUP} the first summand of \eqref{S3} and the second summand of \eqref{S6} sum to
\begin{equation}\label{S31}
-\frac{K_n(w,w)}{(1-|z|^2)^2}|(\varphi_{n+1}^*)^{\prime}(z)\varphi_{n+1}(z)-\varphi_{n+1}^*(z)\varphi_{n+1}^{\prime}(z)|^2.
\end{equation}
Appealing to \eqref{CDUP}, the remaining summand of \eqref{S3} simplifies as
\begin{align}
\nonumber
(|\varphi_{n+1}^{\prime}(z)|^2-&|(\varphi_{n+1}^*)^{\prime}(z)|^2)\\
\label{S32}
&\ \cdot \frac{|\varphi_{n+1}^*(z)\varphi_{n+1}^*(w)|^2+|\varphi_{n+1}(z)\varphi_{n+1}(w)|^2-2\textup{Re}\left(\varphi_{n+1}^*(z)\overline{\varphi_{n+1}^*(w)\varphi_{n+1}(z)}\varphi_{n+1}(w) \right)}{(1-|z|^2)|1-z\overline{w}|^2},
\end{align}
the remaining summand of \eqref{S6} reduces to
\begin{align}
\nonumber
&\frac{2K_n(z,z)\textup{Re}\left(\overline{(\varphi_{n+1}^*)^{\prime}(z)}\varphi_{n+1}^*(w)\varphi_{n+1}^{\prime}(z)\overline{\varphi_{n+1}(w)} \right)}{|1-z\overline{w}|^2}\\
\nonumber
&\quad -\frac{1}{(1-|z|^2)|1-z\overline{w}|^2}\Big[|(\varphi_{n+1}^*)^{\prime}(z)|^2\left(|\varphi_{n+1}^*(z)\varphi_{n+1}^*(w)|^2-|\varphi_{n+1}(z)\varphi_{n+1}^*(w)\right)\\
\label{S33}
&\qquad \qquad \qquad \qquad \qquad \quad  +|\varphi_{n+1}(z)|^2\left(|\varphi_{n+1}^*(z)\varphi_{n+1}(w)|^2-|\varphi_{n+1}(z)\varphi_{n+1}(w)\right)\Big],
\end{align}
and \eqref{S9} can be written as
\begin{align}
\nonumber
&\frac{2|(\varphi_{n+1}^*)^{\prime}(z)|^2}{(1-|z|^2)|1-z\overline{w}|^2}\left(|\varphi_{n+1}^*(z)\varphi_{n+1}^*(w)|^2-\textup{Re}\left( \overline{\varphi_{n+1}^*(z)}\varphi_{n+1}(z)\varphi_{n+1}^*(w)\overline{\varphi_{n+1}(w)}\right)  \right)\\
\nonumber
&\ - \frac{2|\varphi_{n+1}^{\prime}(z)|^2}{(1-|z|^2)|1-z\overline{w}|^2}\left(|\varphi_{n+1}(z)\varphi_{n+1}(w)|^2-\textup{Re}\left( \overline{\varphi_{n+1}^*(z)}\varphi_{n+1}(z)\varphi_{n+1}^*(w)\overline{\varphi_{n+1}(w)}\right)  \right)\\
\nonumber
& \ -\frac{2\left(|\varphi_{n+1}^*(w)|^2-|\varphi_{n+1}(w)|^2 \right)\textup{Re}\left( \overline{(\varphi_{n+1}^*)^{\prime}(z)}\varphi_{n+1}^*(z)\varphi_{n+1}^{\prime}(z)\overline{\varphi_{n+1}(z)}\right)}{(1-|z|^2)|1-z\overline{w}|^2}\\
\label{S34}
&\ -\frac{2K_n(z,z)\textup{Re}\left(\overline{(\varphi_{n+1}^*)^{\prime}(z)}\varphi_{n+1}^*(w)\varphi_{n+1}^{\prime}(z)\overline{\varphi_{n+1}(w)} \right)}{|1-z\overline{w}|^2}.
\end{align}
Simplifying the sum of expressions \eqref{S32}, \eqref{S33}, and \eqref{S34}, then combining with \eqref{S31} we achieve
\begin{align}
\nonumber
\Sigma_{n,3}(z,w)&=\frac{K_n(w,w)|(\varphi_{n+1}^*)^{\prime}(z)\varphi_{n+1}(z)-\varphi_{n+1}^*(z)\varphi_{n+1}^{\prime}(z)|^2}{1-|z|^2}
\left(\frac{1-|w|^2}{|1-z\overline{w}|^2}-\frac{1}{1-|z|^2} \right)\\
\label{Sig3}
&=o(|\phi_{n+1}(z)|^4|\phi_{n+1}(w)|^2).
\end{align}
Thus the sum of \eqref{Sig1}, \eqref{Sig2}, and \eqref{Sig3}, combine to give
\begin{align}
\label{ftilde}
\tilde{f}_n(z,w)&=|\phi_{n+1}(z)|^4|\phi_{n+1}(w)|^2\left( \frac{|z-w|^4}{(1-|z|^2)^4(1-|w|^2)|1-z\overline{w}|^4}+o(1) \right).
\end{align}

The calculations for $\tilde{g}_n(z,w)$ are done in a similar fashion as for $\tilde{f}(z,w)$.  Due the complication nature of the computations, we see fit to include the complete derivation of $\tilde{g}_n(z,w)$.  Let
$$S_4(z,w)=\eqref{s4}, \quad S_5(z,w)=\eqref{s5}, \quad S_6(z,w)=\eqref{s6}.$$

From \eqref{k01} and \eqref{k11} it follows that
\begin{align}
\label{S41}
S_{4}(z,w)&=\frac{(1+z\overline{w})K_n(z,w)}{(1-z\overline{w})^2}(K_n(z,z)K_n(w,w)-|K_n(z,w)|^2)\\
\label{S42}
&\quad +\frac{z\overline{S_n(w,z)}+\overline{w}S_n(z,w)}{(1-z\overline{w})^2}(K_n(z,z)K_n(w,w)-|K_n(z,w)|^2)\\
\label{S43}
&\quad + \frac{R_n(z,w)}{1-z\overline{w}}(K_n(z,z)K_n(w,w)-|K_n(z,w)|^2),\\
\label{S51}
S_5(z,w)&=-\frac{|z|^2K_n(z,z)K_n(w,w)K_n(z,w)}{(1-|z|^2)(1-z\overline{w})}-\frac{|w|^2K_n(z,z)K_n(w,w)K_n(z,w)}{(1-|w|^2)(1-z\overline{w})}\\
\label{S52}
&\quad -\frac{K_n(z,w)}{1-z\overline{w}}\left( \frac{zK_n(w,w)\overline{S_n(z,z)}}{1-|z|^2}+\frac{\overline{w}K_n(z,z)S_n(w,w)}{1-|w|^2} \right)\\
\label{S53}
&\quad -\frac{K_n(z,z)K_n(w,w)}{1-z\overline{w}}\left( \frac{\overline{z}S_n(z,w)}{1-|z|^2}+\frac{w\overline{S_n(w,z)}}{1-|w|^2}\right)\\
\label{S54}
&\quad -\frac{K_n(w,w)\overline{S_n(z,z)}S_n(z,w)}{(1-|z|^2)(1-z\overline{w})}-\frac{K_n(z,z)S_n(w,w)\overline{S_n(w,z)}}{(1-|w|^2)(1-z\overline{w})},\\
\label{S61}
S_6(z,w)&=\frac{\overline{z}wK_n(z,z)K_n(w,w)K_n(z,w)}{(1-|z|^2)(1-|w|^2)}+\frac{z\overline{w}K(z,w)|K(z,w)|^2}{(1-z\overline{w})^2}\\
\label{S62}
&\quad +K(z,w)\left( \frac{wK_n(w,w)\overline{S_n(z,z)}+\overline{z}K_n(z,z)S_n(w,w)}{(1-|z|^2)(1-|w|^2)} \right)\\
\label{S63}
&\quad +K(z,w)\left( \frac{\overline{wK_n(z,w)}S_n(z,w)+z\overline{K_n(z,w)S_n(w,z)}}{(1-z\overline{w})^2} \right)\\
\label{S64}
&\quad + \frac{K_n(z,w)\overline{S_n(z,z)}S_n(w,w)}{(1-|z|^2)(1-|w|^2)}+\frac{\overline{K_n(z,w)S_n(w,z)}S_n(z,w)}{(1-z\overline{w})^2}.
\end{align}

We now set
\begin{align*}
\Sigma_{n,4}(z,w)&:=\eqref{S41}+\eqref{S51}+\eqref{S61}\\
\Sigma_{n,5}(z,w)&:= \eqref{S42}+\eqref{S52}+\eqref{S53}+\eqref{S62}+\eqref{S63}\\
\Sigma_{n,6}(z,w)&:= \eqref{S43}+\eqref{S54}+\eqref{S64}.
\end{align*}

Simplifying then using the relations \eqref{CDUP}, \eqref{phiasy}, and \eqref{dphiasy} yields
\begin{align}
\nonumber
\Sigma_{n,4}(z,w)&=K_n(z,z)K_n(w,w)K_n(z,w)\left(\frac{1+\overline{z}(w-2z)+\overline{w}(z-2w+w|z|^2)}{(1-|z|^2)(1-|w|^2)(1-z\overline{w})^2}  \right)
 -\frac{K_n(z,w)|K_n(z,w)|^2}{(1-z\overline{w})^2}\\
\nonumber
&=|\phi_{n+1}(z)|^2|\phi_{n+1}(w)|^2\phi_{n+1}(z)\overline{\phi_{n+1}(w)}\\
\nonumber
&\quad \cdot \left( \frac{1+\overline{z}(w-2z)+\overline{w}(z-2w+w|z|^2)}{(1-|z|^2)^2(1-|w|^2)^2(1-z\overline{w})^3}-\frac{1}{(1-z\overline{w})^4(1-\overline{z}w)} +o(1)\right)\\
\label{Sign4}
&=|\phi_{n+1}(z)|^2|\phi_{n+1}(w)|^2\phi_{n+1}(z)\overline{\phi_{n+1}(w)}\left( \frac{-|z-w|^4}{(1-|z|^2)^2(1-|w|^2)^2(1-z\overline{w})^4(1-\overline{z}w)}+o(1) \right).
\end{align}

To begin calculating $\Sigma_{n,5}$, first we combine \eqref{S42}, \eqref{S53}, and \eqref{S63}, to give
\begin{equation}
\label{Sign51}
\frac{K_n(z,z)K_n(w,w)}{(1-z\overline{w})^2}\left( \frac{(z-w)\overline{S_n(w,z)}}{1-|w|^2}-\frac{(\overline{z}-\overline{w})S_n(z,w)}{1-|z|^2} \right).
\end{equation}
Now combing \eqref{S52} with \eqref{S62} yields
\begin{align}
\label{Sign52}
&\frac{K_n(z,w)K_n(w,w)\overline{S_n(z,z)}}{1-|z|^2}\left( \frac{w}{1-|w|^2}-\frac{z}{1-z\overline{w}} \right)+\frac{K_n(z,w)K_n(z,z)S_n(w,w)}{1-|w|^2}\left( \frac{\overline{z}}{1-|z|^2}-\frac{\overline{w}}{1-z\overline{w}} \right).
\end{align}

Summing \eqref{Sign51} with \eqref{Sign52} then using \eqref{CDUP}, \eqref{k01}, \eqref{k11} to simplify the expressions, and finally appealing the asymptotics of \eqref{phiasy} and \eqref{dphiasy} on sees
\begin{align}
\nonumber
\Sigma_{n,5}(z,w)&=\frac{1}{(1-|z|^2)(1-|w|^2)(1-z\overline{w})^2}\\
\nonumber
&\cdot\Big[(z-w)K_n(w,w)(\varphi_{n+1}(z)(\varphi_{n+1}^*)^{\prime}(z)-\varphi_{n+1}^*(z)\varphi_{n+1}^{\prime}(z))\\
\nonumber
&\quad \quad  \cdot(\overline{\varphi_{n+1}^*(z)\varphi_{n+1}(w)}-\overline{\varphi_{n+1}(z)\varphi_{n+1}^*(w)}  )\\
\nonumber
&\quad +(\overline{z}-\overline{w})K_n(z,z)(\overline{\varphi_{n+1}(w)(\varphi_{n+1}^*)^{\prime}(w)}-\overline{\varphi_{n+1}^*(w)\varphi_{n+1}^{\prime}(w)})\\
\nonumber
&\quad \quad  \cdot(\varphi_{n+1}^*(z)\varphi_{n+1}(w)-\varphi_{n+1}(z)\varphi_{n+1}^*(w)  )\Big]\\
\label{Sign5end}
&=o(|\phi_{n+1}(z)|^2|\phi_{n+1}(w)|^2\phi_{n+1}(z)\overline{\phi_{n+1}(w)}).
\end{align}

Turning now to $\Sigma_{n,6}(z,w)$, combining the first summands of \eqref{S43} and \eqref{S54} and using \eqref{CDUP}, \eqref{k01}, \eqref{k11} to further simplify gives
\begin{align}
\label{Sign61}
\frac{K_n(w,w)}{(1-|z|^2)(1-z\overline{w})}&((\varphi_{n+1}^*)^{\prime}(z)\varphi_{n+1}(z)-\varphi_{n+1}^*(z)\varphi_{n+1}^{\prime}(z))(\overline{\varphi_{n+1}^{\prime}(w)\varphi_{n+1}^*(z)}-\overline{\varphi_{n+1}(z)(\varphi_{n+1}^*)^{\prime}(w)}).
\end{align}
From the relations \eqref{CDUP}, \eqref{k01}, and \eqref{k11}, the sum of second summands of \eqref{S43} and \eqref{S64} reduces to
\begin{align}
\label{Sign62}
\frac{\overline{K_n(z,w)}}{(1-z\overline{w})^2}&(\varphi_{n+1}(z)(\varphi_{n+1}^*)^{\prime}(z)-\varphi_{n+1}^{\prime}(z)\varphi_{n+1}^*(z))
(\overline{(\varphi_{n+1}^*)^{\prime}(w)\varphi_{n+1}(w)}-\overline{\varphi_{n+1}^*(w)\varphi_{n+1}^{\prime}(w)}).
\end{align}

Using \eqref{CDUP}, \eqref{k01}, and \eqref{k11}, second summand of \eqref{S54} with first summand of \eqref{S64} combine to yield
\begin{align}
\nonumber
\frac{(1-\overline{z}w)\overline{K_n(z,w)}}{(1-|z|^2)(1-|w|^2)(1-z\overline{w})}&(\varphi_{n+1}(z)(\varphi_{n+1}^*)^{\prime}(z)-\varphi_{n+1}^*(z)\varphi_{n+1}^{\prime}(z))\\
\label{Sign63}
& \quad \cdot (\overline{\varphi_{n+1}^*(w)\varphi_{n+1}^{\prime}(w)}-\overline{\varphi_{n+1}(w)(\varphi_{n+1}^*)^{\prime}(w)}).
\end{align}

Combining \eqref{Sign61}, \eqref{Sign62}, and \eqref{Sign63}, then again appealing \eqref{CDUP}, \eqref{k01}, \eqref{k11} to simplify the expressions, and using the asymptotics of \eqref{phiasy} and \eqref{dphiasy} it follows that
\begin{align}
\nonumber
\Sigma_{n,6}(z,w)&=\frac{\overline{K_n(z,w)}}{1-z\overline{w}}((\varphi_{n+1}^*)^{\prime}(z)\varphi_{n+1}(z)-\varphi_{n+1}^*(z)\varphi_{n+1}^{\prime}(z))
(\overline{\varphi_{n+1}^*(w)\varphi_{n+1}^{\prime}(w)}-\overline{\varphi_{n+1}(w)(\varphi_{n+1}^*)^{\prime}(w)})\\
\nonumber
& \quad \cdot \left(\frac{1-\overline{z}w}{(1-|z|^2)(1-|w|^2)}-\frac{1}{1-z\overline{w}} \right)\\
\label{Sign6end}
&=o(|\phi_{n+1}(z)|^2|\phi_{n+1}(w)|^2\phi_{n+1}(z)\overline{\phi_{n+1}(w)}).
\end{align}

Summing \eqref{Sign4}, \eqref{Sign5end}, and \eqref{Sign6end}, we see that
\begin{align}
\label{gtilde}
\tilde{g}_n(z,w)&=|\phi_{n+1}(z)\phi_{n+1}(w)|^2\phi_{n+1}(z)\overline{\phi_{n+1}(w)}
\left( \frac{-|z-w|^4}{(1-|z|^2)^2(1-|w|^2)^2(1-z\overline{w})^4(1-\overline{z}w)}+o(1) \right).
\end{align}

Combining \eqref{fsecint3}, \eqref{rhodeom}, \eqref{ftilde}, and \eqref{gtilde}, we therefore achieve
\begin{align*}
\pi^2\rho_n^{(2)}(z,w)&=\left(|\phi_{n+1}(z)\phi_{n+1}(w)|^6\left(\left(\frac{1}{(1-|z|^2)(1-|w|^2)}-\frac{1}{|1-z\overline{w}|^2} \right)^3+o(1)\right)\right)^{-1}\\
&\quad \cdot \Bigg[|\phi_{n+1}(z)\phi_{n+1}(w)|^6\left( \frac{|z-w|^4}{(1-|z|^2)^4(1-|w|^2)|1-z\overline{w}|^4}+o(1) \right) \\
&\quad \quad \quad \cdot \left( \frac{|z-w|^4}{(1-|w|^2)^4(1-|z|^2)|1-z\overline{w}|^4}+o(1) \right)\\
&\quad \quad +|\phi_{n+1}(z)\phi_{n+1}(w)|^6
\left( \frac{-|z-w|^4}{(1-|z|^2)^2(1-|w|^2)^2(1-z\overline{w})^4(1-\overline{z}w)}+o(1) \right)\\
&\quad \quad \quad \cdot 
\left( \frac{-|z-w|^4}{(1-|z|^2)^2(1-|w|^2)^2(1-\overline{z}w)^4(1-z\overline{w})}+o(1) \right)\Bigg]\\
&=\frac{1}{(1-|z|^2)^2(1-|w|^2)^2}-\frac{1}{|1-z\overline{w}|^4}+o(1),\quad  \text{as} \quad n \rightarrow \infty.
\end{align*}
\end{proof}

\begin{proof}[Proof of Theorem \ref{VAROPS}]
Appealing to  \eqref{varform} and \eqref{fstinop} gives
\begin{align}
\nonumber
\lim_{n\rightarrow \infty}&\textup{Var}[N_n(A(s,t))]=\frac{1}{\pi}\int_{A(s,t)}\frac{1}{(1-|z|^2)^2} dA(z)\\
\nonumber
&\quad +\frac{1}{\pi^2}\int_{A(s,t)}\int_{A(s,t)}\left(\frac{1}{(1-|z|^2)^2(1-|w|^2)^2}-\frac{1}{|1-z\overline{w}|^4}\right) dA(z) dA(w)\\
\nonumber
&\quad -\frac{1}{\pi^2}\int_{A(s,t)}\int_{A(s,t)}\frac{1}{(1-|z|^2)^2(1-|w|^2)^2}dA(z) dA(w)\\
\label{lfintint}
&=\frac{1}{\pi}\int_{A(s,t)}\frac{1}{(1-|z|^2)^2} dA(z)\\
\label{ssecint2}
&\quad -\frac{1}{\pi^2}\int_{A(s,t)}\int_{A(s,t)}\frac{1}{(1-z\bar{w})^2(1-\bar{z}w)^2} dA(z)  dA(w),
\end{align}
where we have used that the convergence of \eqref{fstinop} and \eqref{secintop} are locally uniform on annuli that do not contain the unit circle so that we can pass the limit over the integration.  We remark that formally we need to consider a closed annulus the does not contain the unit circle in the above.  However, since the measure associated to the integrals is Lebesgue area measure which is absolutely continuous, and the limiting values above are continuous functions away from the unit circle, we have that the boundary of the closed annulus has measure zero.  Hence we just consider the open annulus $A(s,t)$.

Observe that for an annulus $A(s,t)$ not containing the unit circle, \eqref{lfintint} simply integrates as
\begin{align}\label{lfintint2}
\frac{1}{\pi}\int_{A(s,t)}\frac{1}{(1-|z|^2)^2} \ dz
=\frac{t^2-s^2}{(1-t^2)(1-s^2)}.
\end{align}

We will compute \eqref{ssecint2} separately depending on whether or not the annulus is contained within the unit disk.  We first consider the case when the annulus is contained with in the unit disk.  Since for $x,y\in \D$ we have
$$\frac{1}{(1-xy)^2}=\sum_{n=0}^{\infty}(n+1)(xy)^n,$$
switching to polar coordinates with $z=re^{i\theta}$ and $w=ue^{i\phi}$, the double integral \eqref{ssecint2} becomes
\begin{equation}\label{secint21}
\frac{1}{\pi^2}\int_0^{2\pi}\int_s^t\int_0^{2\pi}\int_s^t\left(\sum_{k=0}^{\infty}(k+1)(ru)^ke^{ik(\theta-\phi)}\sum_{n=0}^{\infty}(n+1)(ru)^ne^{in(\phi-\theta)}\right)ru dr d\theta du d\phi.
\end{equation}
By the convergence of the above series being locally uniform in the unit disk, we can interchange the infinite sums and integrals.  After expanding the product of sums, any sum with $n\neq k$ will give
\begin{equation*}
\int_0^{2\pi}e^{i(k-n)\theta}d\theta=0,
\end{equation*}
and similarly for the terms with $\phi$. Thus only the terms $n=k$ will survive the integration.  This yields \eqref{secint21} is now
\begin{align*}
4\int_s^t\int_s^t\sum_{k=0}^{\infty}(k+1)^2(ru)^{2k+1} dr du&=2\int_s^t\sum_{k=0}^{\infty}(k+1)u^{2k+1}(t^{2k+2}-s^{2k+2})du\\
&=\sum_{k=0}^{\infty}(t^{2k+2}-s^{2k+2})^2\\
&=\frac{t^4}{1-t^4}-2\frac{(st)^{2}}{1-(st)^2}+\frac{s^4}{1-s^4}\\
&=\frac{(t^2-s^2)^2(1+(st)^2)}{(1-t^4)(1-s^4)(1-(st)^2)}.
\end{align*}
Therefore, combining the above with \eqref{lfintint2} we see that for $A(s,t)\subsetneq \D$
\begin{align*}
\lim_{n\rightarrow \infty}\textup{Var}[N_n(A(s,t))]&=\frac{1}{\pi}\int_{A(s,t)}\frac{1}{(1-|z|^2)^2}\ dz \\
&\quad-\frac{1}{\pi^2}\int_{A(s,t)}\int_{A(s,t)}\frac{1}{(1-z\bar{w})^2(1-\bar{z}w)^2}\ dz \ dw\\
&=\frac{t^2-s^2}{(1-t^2)(1-s^2)}-\left( \frac{(t^2-s^2)^2(1+(st)^2)}{(1-t^4)(1-s^4)(1-(st)^2)} \right)\\
&=\frac{(t^2-s^2)[1-s^2(t^4(2+s^2)-2)]}{(1-t^4)(1-s^4)(1-(st)^2)}.
\end{align*}

When $A(s,t)\subset \C\setminus \overline{\D}$, notice that for $x,y\in \C\setminus \overline{\D}$ it follows that
$$\frac{1}{(1-xy)^2}=\frac{1}{(xy)^2(1-(xy)^{-1})^2}=\frac{1}{(xy)^2}\sum_{n=0}^{\infty}(n+1)(xy)^{-n}.$$
Thus setting $x=z=re^{i \theta}$ and $y=w=ue^{i\phi}$ we have
\begin{align*}
\frac{1}{(1-z\overline{w})^2}\frac{1}{(1-\overline{z}w)^2}=\frac{1}{(ru)^4}\sum_{k=0}^{\infty}(k+1)(ru)^{-k}e^{ik(\phi-\theta)}\sum_{n=0}^{\infty}(n+1)(ru)^{-n}e^{in(\theta-\phi)}.
\end{align*}
As in the case in the unit disk, due to the orthogonality of the exponential function, within the integral of \eqref{ssecint2} only the terms when $n=k$ will give a nonzero integration to hence yield that \eqref{ssecint2} is
\begin{align*}
4\int_s^t\int_s^t\sum_{k=0}^{\infty}(k+1)^2(ru)^{-2k-3} dr du&=-2\int_s^t\sum_{k=0}^{\infty}(k+1)u^{-2k-3}(t^{-2k-2}-s^{-2k-2})du\\
&=\sum_{k=0}^{\infty}(t^{-2k-2}-s^{-2k-2})^2\\
&=\frac{t^4}{1-t^4}-2\frac{(st)^{2}}{1-(st)^2}+\frac{s^4}{1-s^4}\\
&=\frac{-(t^2-s^2)^2(1+(st)^2)}{(1-t^4)(1-s^4)(1-(st)^2)}.
\end{align*}
Combining the above with \eqref{lfintint2}, we conclude that for $A(s,t)\subset \C\setminus \overline{\D}$ it follows that
\begin{align*}
\lim_{n\rightarrow \infty}\textup{Var}[N_n(A(s,t))]&=\frac{1}{\pi}\int_{A(s,t)}\frac{1}{(1-|z|^2)^2}\ dz\\
&\quad  -\frac{1}{\pi^2}\int_{A(s,t)}\int_{A(s,t)}\frac{1}{(1-z\bar{w})^2(1-\bar{z}w)^2}\ dz \ dw\\
&=\frac{t^2-s^2}{(1-t^2)(1-s^2)}-\left( \frac{-(t^2-s^2)^2(1+(st)^2)}{(1-t^4)(1-s^4)(1-(st)^2)} \right)\\
&=\frac{t^2-s^2}{(1-t^2)(1-s^2)}+ \frac{(t^2-s^2)^2(1+(st)^2)}{(1-t^4)(1-s^4)(1-(st)^2)}\\
&=\frac{(t^2-s^2)[1-t^2(s^4(2+t^2)-2)]}{(1-t^4)(1-s^4)(1-(st)^2)}.
\end{align*}

\end{proof}

\textbf{Acknowledgements.}  The author would like thank his P.h.D. advisor Igor Pritsker for initiating this project, as well as for the many helpful discussions that led to the results.

\Addresses

\end{document}